\documentclass[a4paper,11pt]{article}

\input{styleart.sty}
\date{}

\begin{document}

\author{Rizos Sklinos}
\title{On the Generic Type of the Free Group\footnote{The copyright is held by the Association for Symbolic logic}}
\maketitle

\begin{abstract}
We answer a question raised in \cite{PillayGenericity}, that is whether the infinite weight of the generic type of the free group is witnessed 
in $F_{\omega}$. We also prove that the set of primitive elements in finite rank free groups is not uniformly definable. As 
a corollary, we observe that the generic type over the empty set is not isolated. Finally, we show that uncountable free groups are not 
$\aleph_1-$homogeneous. 
\end{abstract}

\section{Introduction}
As pointed out in \cite{PillayForking}, the free group is connected 
and thus it has a unique generic type over any set 
of parameters. In particular, there is a unique generic type over the empty set, 
which we denote as $p_0$. Furthermore, $p_0$ has some nice (or not so nice) properties. In \cite{PillayGenericity}, it 
was shown that $p_0$ has infinite weight. On the way to proving this, it was also shown that 
the realizations in $F_n$ of $p_0$ are exactly the primitives. 

In this paper we mainly explore basic model theoretic properties of $p_0$ and
from them we deduce some useful facts about the free group. In the remainder
of this section, we give some quick background, and definitions of basic notions
around the free groups. In section \ref{NonIsol}, we prove the non uniform definability of the primitives in the
finite rank free groups, and deduce non isolation of $p_0$ and other corollaries. 
Finally, in section \ref{Weight}, we answer some questions raised by Pillay in \cite{PillayGenericity}, 
including whether the infinite weight of $p_0$ is witnessed in $F_{\omega}$. 
We also show that $F_{\kappa}$, for $\kappa>\omega$, is not $\aleph_1-$homogeneous (as a structure).

We will freely use notions from stability theory, such as forking, independent sequence, weight, etc. Also 
notions from stable group theory, such as generic type, generic set, connected group, etc. Our main 
reference for stability is \cite{PillayStability}. Stable groups are studied elegantly, but still in great depth in \cite{PoizatStableGroups}. For the 
unfamiliar reader, there is a quick, dense introduction in \cite{PillayForking},\cite{PillayGenericity}. 

Finally, I would like to thank Zlil Sela and Julia Knight. Zlil Sela for a useful discussion we had at a conference in Southampton, 
and Julia Knight for bringing to my attention paper \cite{MartinoVentura} on the references, thus the Whitehead graph technique. 
Special thanks go to my thesis supervisor, Anand Pillay, for his constant help in both the 
preparation and the subject matter of this paper.

\subsection{Free Groups}
Let us now give some basic facts about the free groups. Let $F_n$ denote the free group on $n$ generators. If 
$E=\{e_1,\ldots,e_n\}$ is the set of generators, we can identify $F_n$ with the set of reduced words in $E\cup E^{-1}$, 
and group operation concatenation followed by reductions (such that the result is a reduced word). 
A word $w=u_1u_2\ldots u_k$, with $u_i\in E\cup E^{-1}$ is {\em reduced} if for all $1\leq i\leq k-1$, $u_i\neq u_{i+1}^{-1}$. 
Moreover, $w$ is {\em cyclically reduced} if it is reduced and $u_1\neq u_k^{-1}$.  
An element of $F_n$ is called {\em primitive} if it belongs to some basis of $F_n$. It is quite clear that primitive elements 
form a single orbit under $Aut(F_n)$. In \cite{PillayGenericity} the following fact was observed.
\begin{fact}
Let $\{e_1,\ldots,e_n\}$ be a basis of $F_n$. Let $m\leq n$, and $k_1,\ldots,k_m$ integers $>1$. Then 
$e_1^{k_1}\ldots e_m^{k_m}$ is not a primitive of $F_n$. 
\end{fact}
\begin{proof}(Sketch)
$Aut(F_n)$ is generated by Whitehead automorphisms, and no Whitehead automorphism can reduce 
the length of $w=e_1^{k_1}\ldots e_m^{k_m}$. Thus, $w$ and $e_1$ cannot be in the same orbit 
under $Aut(F_n)$.
\end{proof}

For completeness, we give the definition of a Whitehead automorphism, 
as given in \cite[p.31]{LyndonSchupp}.

\begin{definition}
Let $F$ be a free group generated by $X$, then $\tau$ is a Whitehead automorphism 
of $F$ if it is an automorphism of one of the following two kinds
\begin{enumerate}
 \item $\tau$ permutes the elements of $X^{\pm 1}$
 \item for some fixed ``multiplier'' $a\in X^{\pm 1}$, $\tau$ carries 
 each of the elements $x\in X^{\pm 1}$ into one of 
 $x, xa, a^{-1}x,$ or $a^{-1}xa$. 
\end{enumerate}

\end{definition}

We note that in the definition above we require that $\tau$ is an automorphism. Not 
all maps satisfying $(1)$ or $(2)$ are automorphisms.
\\ \\
Tarski around 1945 posed the following question, 
``Do the free groups in more than two generators have the same common theory?''
\\ \\
In \cite{Sel6} Sela proved.
 
\begin{theorem}[Sela]
If $2\leq m\leq n$ then the natural embedding of $F_m$ to $F_n$ is an elementary embedding. 
\end{theorem}

As this answers Tarski's question, we are now free to denote the theory of the free group as $T_{fg}$. 
\\ \\
Sela has also proved in \cite{SelaStability} the following rather astonishing result.

\begin{theorem}[Sela]
$T_{fg}$ is stable.
\end{theorem}

Let us also remark here that, from previous work of Poizat \cite{PoizatSuper}, $T_{fg}$ is not superstable.

\section{Non isolation of the generic type} \label{NonIsol}
Our aim in this section is to prove that the primitives are not uniformly definable in finite rank free groups. 
The non isolation of $p_0$ would only be an easy corollary. A basic step, on the way to proving this, is 
to show that the set of non primitives, in any free group, is ``big''. 
In the case of a definable set, this simply means generic.

For the benefit of the reader we explain a few things about the model theoretic
point of view on groups.

A group, $(G,\cdot)$, in the sense of model theory is a structure 
equipped with a group operation, but possibly also with some additional
relations and functions. Even when we do not have explicitly any additional
relations or functions, all the definable subsets, $X\subseteq G^n$, of the group under
consideration will be part of our structure. In the case that all additional
structure is definable by multiplication alone, we speak of a pure group.

We define a stable group to be a group definable in a stable theory. By this we
mean that $(G,\cdot)$ is definable in a model $M$ of the stable theory $T$, and it may be
equipped with some or all of the structure induced from $M$. The simplest case
will be when the group coincides with the ambient structure and the underlying
language is the language of groups. Indeed, this is the case for non abelian free
groups.
\begin{definition}
Let $G$ be a stable group. Let $X$ be a definable subset of $G$.
Then $X$ is left (right) generic if finitely many left (right) translates of $X$, by
elements of $G$, cover $G$.
\end{definition}

As for a definable $X\subseteq G$, $X$ is left generic iff $X$ is right generic, we simply say
generic.

\begin{definition}
Let $G$ be a group. Then $G$ is connected if there is no definable
proper subgroup of finite index.
\end{definition}

Let us note here that connectedness passes to elementarily equivalent groups,
so with an abuse of language we can say that $T_{fg}$ is connected, meaning that
all models of $T_{fg}$ are connected groups.

\begin{definition}
Let $G$ be a stable group. Let $g\in G$, and $A$ a set of parameters
from $G$. Then $tp(g/A)$ is a generic type if every formula in $tp(g/A)$ is generic.
\end{definition}

We recall some useful facts about genericity and connectedness.

\begin{fact}
Let $G$ be a stable group. Let $X$, $Y$ definable subsets of $G$. Then
\begin{itemize}
\item[(i)] if $X\cup Y$ is generic, then one of $X$, $Y$ is generic.
\item[(ii)] $G$ is connected iff there is no definable $X\subseteq G$ such that both $X$ and $G\setminus X$
are generic.
\item[(iii)] $G$ is connected iff there is over any set of parameters a unique generic type
of an element of $G$.
\end{itemize}
\end{fact}

The next important notion is that of a Whitehead graph.

\begin{definition}
Let $a=u_1\ldots u_k$, a word in $F_n=\langle e_1,\ldots,e_n \rangle$. The Whitehead Graph of $a$, $W_a$, is the graph 
with set of vertices, $V(W_a)=\{e_1,\ldots,$ $e_n,e_1^{-1},\ldots,e_n^{-1}\}$, and 
edges joining $u_1$ to $u_2^{-1}$, $u_2$ to $u_3^{-1}$,$\ldots$, $u_{k-1}$ to $u_k^{-1}$,and $u_k$ to $u_1^{-1}$. 
\end{definition}
We note here that the number of edges of $W_a$ equals the length of $a$.
\begin{definition}
Let $G$ be a graph. Then $G$ has a cut vertex, if there is a
vertex, $u$, such that removing $u$ and its adjacent edges leaves the graph disconnected.
\end{definition}

We now give some examples of Whitehead graphs in $F_2=\langle e_1, e_2 \rangle$. The first
two graphs have cut vertices, while the third does not. Also note that in the
first example the graph is already not connected.
\begin{figure}[ht]
\hfill
\begin{minipage}[t]{.3\textwidth}

\begin{graph}(2.5,2.5)(-1.5,0)

\roundnode{1}(0,2)
\roundnode{2}(1,2)
\roundnode{3}(0,1)
\roundnode{4}(1,1)

\edge{1}{2}
\edge{2}{4}
\edge{1}{4}
\loopedge{3}{90}(.3,.3)
 
\freetext(0,2.5){$e_1$}
\freetext(1,2.5){$e_2$}
\freetext(0,0.5){$e_1^{-1}$}
\freetext(1,0.5){$e_2^{-1}$}

\end{graph}

\caption{$e_1e_2^2e_1^{-1}$}

\end{minipage}
\hfill
\begin{minipage}[t]{.3\textwidth}

\begin{graph}(2.5,2.5)(-1.5,0)

\roundnode{1}(0,2)
\roundnode{2}(1,2)
\roundnode{3}(0,1)
\roundnode{4}(1,1)

\bow{1}{4}{0.1}
\bow{1}{4}{-.1}
\bow{3}{2}{0.1}
\bow{3}{2}{-.1}
\edge{1}{3}
 
\freetext(0,2.5){$e_1$}
\freetext(1,2.5){$e_2$}
\freetext(0,0.5){$e_1^{-1}$}
\freetext(1,0.5){$e_2^{-1}$}

\end{graph}

\caption{$(e_1e_2)^2e_1$}

\end{minipage}
\hfill
\begin{minipage}[t]{.3\textwidth}

\begin{graph}(2.5,2.5)(-1.5,0)

\roundnode{1}(0,2)
\roundnode{2}(1,2)
\roundnode{3}(0,1)
\roundnode{4}(1,1)

\edge{1}{3}
\edge{1}{4}
\edge{3}{2}
\edge{4}{2}

\freetext(0,2.5){$e_1$}
\freetext(1,2.5){$e_2$}
\freetext(0,0.5){$e_1^{-1}$}
\freetext(1,0.5){$e_2^{-1}$}

\end{graph}

\caption{$e_1e_2^2e_1$}
\end{minipage}
\end{figure}
\\ \\ \\ \\ \\ \\
The following proposition is a weaker reformulation of Theorem 2.4 in \cite{Stallings} (see also \cite[Theorem 6]{MartinoVentura}) .
 
\begin{proposition}\label{WhiteheadGraph}
Let $F=F_n$, for some $n\geq 2$. Let $a$ be a cyclically reduced primitive of $F$. Then $W_a$ has a cut vertex.
\end{proposition}

The next lemma will help us prove that the set of non primitives is ``big'' in any free group.
 
\begin{lemma}\label{NPbig}
Let $F_n =\langle e_1,\ldots, e_n \rangle$, for some $n\geq 2$. There is a finite collection of words
$\{w_{ij} : i,j\leq n \}\subset F_n$, such that for any $a\in F_n$, $w_{ij}\cdot a$ is non primitive for some
$i,j\leq n$.
\end{lemma}
\begin{proof}
We will give an explicit description of the $w_{ij}$'s. Let $w$ be the following word 
$e_1^2e_n^2e_1e_2^{-1}e_1$ $e_2e_3^{-1}e_2\ldots e_{n-1}e_n^{-1}e_{n-1}$. 
In pictures, the following graph will be part of the Whitehead graph of $w$.

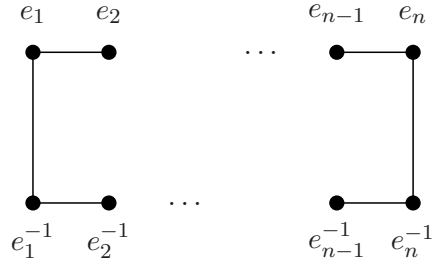
\begin{graph}(5,4)(-3,-.5)

\roundnode{1}(0,2.5)
\roundnode{2}(1,2.5)
\roundnode{3}(4,2.5)
\roundnode{4}(5,2.5)
\roundnode{5}(5,0.5)
\roundnode{6}(4,0.5)
\roundnode{7}(1,0.5)
\roundnode{8}(0,0.5)

\edge{1}{2}
\edge{3}{4}
\edge{4}{5}
\edge{5}{6}
\edge{7}{8}
\edge{1}{8}

\freetext(0,3){$e_1$}
\freetext(1,3){$e_2$}
\freetext(4,3){$e_{n-1}$}
\freetext(5,3){$e_n$}
\freetext(5,0){$e_n^{-1}$}
\freetext(4,0){$e_{n-1}^{-1}$}
\freetext(1,0){$e_2^{-1}$}
\freetext(0,0){$e_1^{-1}$}
\freetext(3,2.5){$\ldots$}
\freetext(2,0.5){$\ldots$}

\end{graph}

In addition, for each $p, q \leq n$, we define $w_{pq}$ to be the word $e_pwe_q$. 
In total we have $n^2$ such words. We show that these words are the $w_{ij}$'s we wanted. 

Let $a\in F_n$, such that
$a$ starts with $e_l^k$ and ends with $e_r^m$, for some $k,m\in \{1,-1\}$ and $l,r\leq n$. 
Then we can choose $w_{ij}$, such that $i\neq r$ and 
$j\neq l$. So $w_{ij}\cdot a$ is cyclically reduced. As there is no cancellation between
$w_{ij}$ and $a$, $W_{w_{ij}a}$ contains the circle pictured above. Thus, $W_{w_{ij}a}$ does not have a
cut vertex. And, by Proposition \ref{WhiteheadGraph}, $w_{ij}a$ is non primitive.

The only case left is when $a=1$, the identity element. By the argument above, each $w_{ij}$ is non primitive. 
Thus, $w_{11}\cdot 1$ is non primitive, and this completes the proof. 
\end{proof}

The following proposition is an easy consequence of Lemma \ref{NPbig}.

\begin{proposition}\label{FinCov}
Let $F = F_n$, for some $n\geq 2$. Then finitely many translates
of the set of non primitives, $N$, cover $F$.
\end{proposition}
\begin{proof}
Let $a\in F$. Then, by Lemma \ref{NPbig}, for some $i,j\leq n$, $w_{ij}a\in N$. Thus,
$\bigcup_{i,j\leq n} w_{ij}^{-1}N$ covers $F$.
\end{proof}
 
We next show that primitives are not uniformly definable. 

\begin{proposition}\label{UnDef}
Let $F_n=\langle e_1,\ldots,e_n \rangle$, for some $n\geq 2$. Then there is no formula, $\phi(x,\bar{y})$, and a set 
of parameters $\bar{b}_n$ in $F_n$, 
such that $\phi(x,\bar{b}_n)$ defines the set of primitives in $F_n$, for each $n\geq 2$.
\end{proposition}

\begin{proof}
Suppose not, and $\phi(x,\bar{y})$, $\{\bar{b}_2, \bar{b}_3,\ldots, \bar{b}_n,\ldots\}$ witness it. We show that $\phi(x,\bar{y})$ has 
the order property in $F_{\omega}$, contradicting directly stability of $T_{fg}$.\\ \\
\textit{Claim.} Let $F_{\omega}=\langle e_i:i<\omega \rangle$. Then $F_{\omega}\models\phi(e_{i},\bar{b}_{j})$ iff $i\leq j$ (for $i,j\geq 2$).\\
\textit{Proof.}\\
$(\Leftarrow)$ Suppose $i\leq j$. Then $F_{j}\models \phi(e_{i},\bar{b}_{j})$. But, $F_{j}\prec F_{\omega}$. 
Therefore $F_{\omega}\models\phi(e_{i},\bar{b}_{j})$.\\
$(\Rightarrow)$ Suppose $i>j$, but $F_{\omega}\models\phi(e_{i},\bar{b}_{j})$. We first show that $e_{i}$ is independent 
from $\bar{b}_{j}$ over $\emptyset$. By \cite[Corollary 2.7(ii)]{PillayForking}, 
$e_{i}$ is independent from $e_1,\ldots,e_{i-1}$ over $\emptyset$. 
Thus $e_{i}$ is independent from $acl(e_1,\ldots,e_{i-1})$ over $\emptyset$. But $\bar{b}_j\subseteq acl(e_1,\ldots,e_{i-1})$. 
Therefore $e_{i}$ is independent from $\bar{b}_{j}$ 
over $\emptyset$. So $tp^{F_{\omega}}(e_{i} / \bar{b}_{j})$ is the unique generic type of $F_{\omega}$ over $\bar{b}_{j}$. And because 
$\phi(x,\bar{b}_{j})\in tp^{F_{\omega}}(e_{i}/ \bar{b}_{j})$, $\phi(x,\bar{b}_{j})$ is generic. But $\lnot\phi(x,\bar{b}_{j})$ 
defines the non primitives in $F_{j}$. So, by Proposition \ref{FinCov}, $\lnot\phi(x,\bar{b}_{j})$ is also generic, contradicting the 
connectedness of $F_{\omega}$.        
\end{proof}

The next theorem is an easy corollary of Proposition \ref{UnDef}.

\begin{theorem} 
The generic type $p_0$ of $T_{fg}$, is non isolated.
\end{theorem}
\begin{proof}
Suppose, for the sake of contradiction, that $p_0$ is isolated and $\phi(x)$ witnesses it. 
Then for every $G\models T_{fg}$, we have $\phi(G)= p_0(G)$. So in
particular, we have $\phi(F_n) = p_0(F_n)$, for any $n\geq 2$. But, by \cite[Theorem
2.1]{PillayGenericity}, $p_0(F_n)$ is exactly the set of primitive elements of $F_n$. So, $\phi(x)$ uniformly 
defines the primitives (without parameters), contradicting Proposition \ref{UnDef}.
\end{proof}

At this point let us mention a result of 
Perin \cite{ThesisPerin}, which will play a central role in the next section, 
but also useful here.

\begin{theorem}[Perin]\label{Perin}
Let $F=F_n$, for some $n\geq 2$. Let $G$ be an elementary substructure of $F$. Then $G$ is a free factor of $F$.
\end{theorem}

The next result was also proved by  
Nies \cite{NiesF2} by slightly different methods. 
Our proof uses the omitting types theorem (see \cite[Theorem 4.2.3,p.125]{MarkerModelTheory}), which we quickly recall. 

\begin{theorem}[Omitting Types Theorem]
Let $\mathcal{L}$ be a countable language, $T$ an $\mathcal{L}-$theory, and $p$ a 
(possibly incomplete) non isolated $n-$type over $\emptyset$. Then, there is a countable 
$\mathcal{M}\models T$ omitting $p$.
\end{theorem}

\begin{corollary}
$T_{fg}$ does not have a prime model.
\end{corollary}
\begin{proof}
Let $\mathcal{A}$ be the prime model of $T_{fg}$, then $\mathcal{A}\prec F_2$. So by Theorem \ref{Perin}, $\mathcal{A}\cong F_2$, therefore 
$\mathcal{A}$ realizes $p_0$. But if the prime model realizes $p_0$, then any model realizes $p_0$. 
As $p_0$ is non isolated, this clearly contradicts the omitting types theorem.
\end{proof}

Let us remark here that, by a result of Nielsen \cite{Nielsen}, the set of primitives of
$F_2$ is definable in $F_2$ (over a set of parameters). More precisely, Nielsen proved
that two elements $a, b \in F_2=\langle e_1, e_2 \rangle$, form a basis of $F_2$ iff $[a, b]$ is a
conjugate of $[e_1, e_2]$ or a conjugate of $[e_2, e_1]$. Thus the following formula defines the primitives
in $F_2 $: 
$$\exists z\exists y([x, y] = [e_1,e_2]^z \lor [x, y] = [e_2, e_1]^z)$$

We now pass to the main part of the paper.

\section{Weight and Homogeneity} \label{Weight}
Before we start we quickly recall the notion of weight. We work in a stable theory $T$, in 
a big saturated model $\mathbb{M}$ (what we usually call the monster model), and $A,B$ denote 
small subsets of $\mathbb{M}$.
\begin{definition}
The preweight of a type $p(x) = tp(a/A)$, $prwt(p)$, is the
supremum of the set of cardinals $\kappa$, for which there exists an $A$-independent set
$\{b_i : i <\kappa\}$, such that $tp(a/Ab_i)$ forks over $A$ for all $i$.\\
The weight of a type $p$, $wt(p)$, is the supremum of $\{prwt(q) : q$ a non forking
extension of $p\}$.
\end{definition}

Using forking calculus one can see that, if $T$ is countable, then for any $a,A$, $wt(a/A)\leq\omega$. One could 
also distil from \cite[Lemma 3.9,p.166]{PillayStability},\cite[Proposition 3.10,p.167]{PillayStability} that, if $wt(a/A) = \omega$, 
then for some $B \supseteq A$, such that $a$ is independent
from $B$ over $A$, there is an infinite independent set, $\{b_i : i < \omega\}$, over $B$, such
that $a$ forks with each $b_i$ over $B$.

This section builds upon the following crucial result, proved in \cite{PillayGenericity}.
\begin{theorem}[Pillay]\label{InfWeight}
The generic type $p_0$ of $T_{fg}$ has infinite weight.
\end{theorem}
The method was to find for every $n\geq 2$, a realization, $g$, of $p_0$ in $F_n$, and an independent set 
of realizations, $b_1,\ldots,b_n$, of $p_0$ in $F_n$, such that $g$ forks with each $b_i$.   
Furthermore, in \cite{PillayGenericity} it was observed that, by a compactness argument, one can find $a\models p_0$ and $(b_i:i<\omega)$ 
an independent sequence of realizations of $p_0$ in a model $G$, such that $a$ forks with each $b_i$ over the empty set.  
Therefore, a natural question is whether we can find such elements in $F_{\omega}$. 

In the rest of the section we show that Theorem \ref{Perin} cannot be extended to include $F_{\omega}$. 
Moreover, we answer in the affirmative the question mentioned above, and finally we show 
that $F_{\kappa}$, for $\kappa>\omega$, is not $\aleph_1-$homogeneous.

We now mention a fact, observed in \cite[Fact 1.9]{PillayGenericity}, that we will use through out the section.
 
\begin{fact}\label{ForkCal}
Let $G$ be a connected stable group. Let $A=\{a_i:i\in I\}$ be an independent set of realizations 
of the generic type of $G$ over the empty set, in $G$. Let $\tau$ be one of the following maps:
\begin{itemize}
\item[(i)]for some permutation $\pi$ of $I$, $\tau(a_i) = a_{\pi(i)}$ or $a^{-1}_{\pi(i)}$.
\item[(ii)]for some fixed ``multiplier'' $a_i\in A$, $\tau$ fixes $a_i$ and carries each of the elements $a_j\in A$
into one of $a_j, a_ja_i, a^{-1}_ia_j$, or $a^{-1}_ia_ja_i$.
\end{itemize}

Then $\tau$ is an elementary map in the sense of $G$. In particular $\{\tau(a_i):i\in I\}$ 
is an independent set of realizations of the generic type of $G$ over the empty set.
\end{fact}

Also, the next result was proved in \cite{PillayGenericity}.

\begin{theorem}[Pillay]\label{MaxInd}
Let $F=F_n$, for some $n\geq 2$. Then every maximal independent sequence 
of realizations of $p_0$ in $F$ is a basis of $F$.
\end{theorem}

Let $\kappa$ be a cardinal (maybe infinite), we denote with $p_0^{(\kappa)}$ the type of $\kappa$ independent realizations of $p_0$. 
As $p_0$ is stationary this is a good definition. 
So, in other words, the above theorem says that if $(a_1,\ldots,a_n)\models p_0^{(n)}$ in $F_n$, 
then $F_n=\langle a_1,\ldots,a_n \rangle$. And every maximal independent set of realizations of $p_0$
in $F_n$ has cardinality $n$. 

One might expect that Theorem \ref{MaxInd} extends to $F_{\omega}$. As a matter 
of fact, it follows from the proof, that every finite independent set 
of realizations of $p_0$ in $F_{\omega}$ extends to a basis of $F_{\omega}$. 
This is not the case for infinite indepedent realizations of $p_0$, as we show:

\begin{lemma} 
There is an independent set of realizations of $p_0$ in $F_{\omega}$, that does not extend to a basis of $F_{\omega}$.
\end{lemma}

\begin{proof}
Let $B=\{e_1e_2^2,e_2e_3^2,\ldots,e_ne_{n+1}^2,\ldots\}=\{b_i:i<\omega\}$. Now, because 
$\langle e_1e_2^2,\ldots, e_ne_{n+1}^2,$ $e_{n+1} \rangle =F_{n+1}$, we have that $e_ne_{n+1}^2$ is independent 
from $\{e_1e_2^2,$ $\ldots, e_{n-1}e_n^2\}$. Therefore, $B$ is an independent set of realizations of $p_0$. 
\\ \\
\textit{Claim I}. $e_1\not\in \langle b_i : i < \omega \rangle$.\\
\textit{Proof}. Suppose not, then we may assume that $e_1\in \langle b_1,\ldots, b_n \rangle$, for some $n$. 
Iterating Fact \ref{ForkCal} we have $b_1b_2\ldots b_{n+1}$ is independent from $b_1, b_2, \ldots , b_n$ over $\emptyset$. 
Thus $b_1b_2 \ldots b_{n+1}$ is independent from $acl(b_1, b_2,\ldots, b_n)$ over $\emptyset$. Therefore, 
$b_1b_2 \ldots b_{n+1}$ is independent from $e_1$ over $\emptyset$. But 
$b_1b_2\ldots b_{n+1}=e_1e_2^3\ldots e_{n+1}^3e_{n+2}^2$, so using Fact \ref{ForkCal} 
$e_1^{-1}\cdot e_1e_2^3\ldots e_{n+1}^3e_{n+2}^2=e_2^3\ldots e_{n+1}^3e_{n+2}^2$ is primitive, a contradiction.
\\ \\
\textit{Claim II}. Let $a\in F_{\omega}$, such that $a\models p_0$. Then $\{a\}\cup B$ is a dependent set.\\
\textit{Proof}. Suppose not, then $\{a, b_i : i < \omega\}$ is an infinite independent set of realizations of $p_0$. 
We may assume that $a\in F_{n+1}$. Thus, $\{a, b_1, . . . , b_n\}$ is a maximal
independent set of realizations of $p_0$ in $F_{n+1}$, so, by Theorem \ref{MaxInd}, a basis of
$F_{n+1}$. But by our assumption $b_{n+1}$ is independent from $a, b_1, \ldots, b_n$ over $\emptyset$.
Thus, $b_{n+1}$ is independent from $acl(a, b_1, \ldots, b_n)$ over $\emptyset$. Therefore, 
$b_{n+1}$ is independent from $e_{n+1}$ over $\emptyset$. And $b_{n+1}=e_{n+1}e_{n+2}^2$, so 
using Fact \ref{ForkCal} $e_{n+1}^{-1}\cdot e_{n+1}e_{n+2}^2=e_{n+2}^2$ is primitive, a contradiction.
\\ \\
Therefore, $B$ is a maximal independent set of realizations of $p_0$ in $F_{\omega}$ that
is not a basis of $F_{\omega}$.

\end{proof}

Now we get the next easy corollary.

\begin{corollary}
There is $G\prec F_{\omega}$, such that $G$ is not a free factor of $F_{\omega}$.
\end{corollary}
\begin{proof}
By the previous lemma, we only need to show that 
$G=\langle e_1e_2^2,e_2e_3^2,$ $\ldots,e_ne_{n+1}^2,\ldots \rangle=\langle b_i : i<\omega \rangle$ 
is an elementary substructure of $F_{\omega}$. First note that 
$tp^G(e_1e_2^2,$ $e_2e_3^2,\ldots,e_ne_{n+1}^2,\ldots)$ $=tp^{F_{\omega}}(e_1e_2^2,e_2e_3^2,\ldots,e_ne_{n+1}^2,\ldots) = 
p_0^{(\omega)}$, this is because $G$ is free with basis $\{e_1e_2^2,e_2e_3^2,$ $\ldots,e_ne_{n+1}^2,\ldots\}$. 
Now the proof is straightforward.
Let $a_1,\ldots,a_m\in G$, $F_{\omega}\models \phi(a_1,\ldots,a_m)$ iff 
$F_{\omega}\models \phi(t_1(b_1,\ldots,b_k),\ldots, t_m(b_1,\ldots,b_k))$ iff $F_{\omega}\models\psi(b_1,\ldots,b_k)$ 
iff $G\models\psi(b_1,\ldots,b_k)$ iff $G\models\phi(a_1,\ldots,a_m)$.
\end{proof}

We now turn to the main question of the section. That is, whether we can find $a\models p_0$ and 
$(b_i:i<\omega)\models p_0^{(\omega)}$ in $F_{\omega}$, such that $a$ forks with each $b_i$ over the empty set. 
As a matter of fact, the next lemma could also serve as an alternative proof of Theorem \ref{InfWeight}. 

\begin{lemma}\label{WitinF}
Let $F_{\omega}=\langle e_i : i < \omega \rangle$. Then there exists $\{b_i : i < \omega\}$ an
independent set of realization of $p_0$ in $F_{\omega}$, 
such that $e_1$ forks with $b_i$ over $\emptyset$, for all $i <\omega$. 
\end{lemma}
\begin{proof}
Let $B=\{e_1e_2^2,e_1e_2^3e_3^2,e_1e_2^3e_3^3e_4^2,\ldots,e_1e_2^3\ldots e_n^3 e_{n+1}^2,\ldots\}$. Then,  
$\langle e_1e_2^2$, $e_1e_2^3e_3^2,e_1e_2^3e_3^3e_4^2,$ $\ldots,e_1e_2^3\ldots e_n^3 e_{n+1}^2,e_{n+1}\rangle=F_{n+1}$. Thus,  
$B$ is an independent set of realizations of $p_0$. 
Furthermore, $e_1$ forks with every element of $B$. Suppose not, then $e_1$ is independent from 
$e_1e_2^3\ldots e_n^3 e_{n+1}^2$, so, by Fact \ref{ForkCal}, $e_1^{-1}e_1e_2^3\ldots e_n^3 e_{n+1}^2=e_2^3\ldots e_n^3 e_{n+1}^2$ 
is primitive, a contradiction.
\end{proof}

Now using the invariance of forking, that is whether or not $a$ forks with $b$ over $C$, depends on $tp(a,b,C)$, we show 
that $F_{\kappa}$, for $\kappa > \omega$, is not $\aleph_1-$homogeneous.

\begin{proposition}
Let $F_{\kappa}=\langle e_i : i<\kappa \rangle$, for some $\kappa>\omega$. Then $F_{\kappa}$ is not $\aleph_1-$homogeneous.
\end{proposition}
\begin{proof}
We first note that, by the proof of Lemma \ref{WitinF}, 
$tp(e_1e_2^2,e_1e_2^3e_3^2,e_1e_2^3e_3^3e_4^2,\ldots,e_1e_2^3\ldots$ $e_n^3 e_{n+1}^2,\ldots)=tp(e_1,e_2,e_3,\ldots,e_n,\ldots)=p_0^{(\omega)}$. 
We next show that if we extend 
$tp(e_1e_2^2,e_1e_2^3e_3^2,$ $e_1e_2^3e_3^3e_4^2,\ldots,e_1e_2^3\ldots e_n^3 e_{n+1}^2,\ldots)$ by adding $e_1$, then 
there is no element, $a\in F_{\kappa}$, such that $tp(e_1,e_1e_2^2,e_1e_2^3e_3^2,e_1e_2^3e_3^3e_4^2,\ldots,e_1e_2^3\ldots$ $e_n^3 e_{n+1}^2,\ldots) = tp(a,e_1,e_2,e_3,\ldots,e_n,\ldots)$. 

Suppose, for the sake of contradiction, that such an element exists. Then 
there is $n<\omega$, such that $a$ is independent from $e_n$ over $\emptyset$ (that is because 
$a\in acl(A)$, where $A$ is a finite subset of $\{e_i : i<\kappa\}$ ). But, by forking invariance, 
as $tp(a,e_n)=tp(e_1,e_1e_2^3\ldots e_n^3 e_{n+1}^2)$, 
we have that 
$e_1$ is independent from $e_1e_2^3\ldots e_n^3 e_{n+1}^2$, which is, as in the proof of Lemma \ref{WitinF}, a contradiction.
\end{proof}

\bibliography{biblio}

\providecommand{\bysame}{\leavevmode\hbox to3em{\hrulefill}\thinspace}
\providecommand{\MR}{\relax\ifhmode\unskip\space\fi MR }
\providecommand{\MRhref}[2]{%
  \href{http://www.ams.org/mathscinet-getitem?mr=#1}{#2}
}
\providecommand{\href}[2]{#2}
\begin{thebibliography}{Mar02}

\bibitem[LS77]{LyndonSchupp}
R.C. Lyndon and P.E. Schupp, \emph{Combinatorial group theory},
  Springer-Verlag, 1977.

\bibitem[Mar02]{MarkerModelTheory}
David Marker, \emph{Model theory: an introduction}, Graduate Texts in
  Mathematics, vol. 217, Springer, 2002.

\bibitem[MV03]{MartinoVentura}
Armando Martino and Enric Ventura, \emph{{Examples of retracts in free groups
  that are not the fixed subgroup of any automorphism}}, J. of Algebra
  \textbf{269} (2003), 735--747.

\bibitem[Nie17]{Nielsen}
Jacob Nielsen, \emph{Die isomorphismen der allgemeinen unendliehen gruppe mit
  zwei erzeugenden}, Mathematische Annalen (1917), 385--397.

\bibitem[Nie03]{NiesF2}
Andre Nies, \emph{Aspects of free groups}, J. Algebra \textbf{263} (2003),
  119--125.

\bibitem[Per08]{ThesisPerin}
Chlo\'e Perin, \emph{Elementary embeddings in torsion-free hyperbolic groups},
  Ph.D. thesis, Universit\'e de Caen Basse-Normandie, October 2008.

\bibitem[Pil96]{PillayStability}
Anand Pillay, \emph{Geometric stability theory}, Oxford University Press, 1996.

\bibitem[Pil08]{PillayForking}
\bysame, \emph{Forking in the free group}, J. Inst. Math. Jussieu \textbf{7}
  (2008), 375--389.

\bibitem[Pil09]{PillayGenericity}
\bysame, \emph{On genericity and weight in the free group}, Proc. Amer.
  Math.Soc. \textbf{137} (2009), 3911--3917.

\bibitem[Poi83]{PoizatSuper}
Bruno Poizat, \emph{{Groupes stables avec types generiques reguliers}}, J. of
  Symbolic Logic \textbf{48} (1983), 641--658.

\bibitem[Poi01]{PoizatStableGroups}
\bysame, \emph{Stable groups}, Mathematical Surveys and Monographs, vol.~87,
  AMS, 2001.

\bibitem[Sel]{SelaStability}
Zlil Sela, \emph{{Diophantine geometry over groups VIII: Stability}},
  {preprint, available at \texttt{http://www.ma.huji.ac.il/~zlil/}}.

\bibitem[Sel06]{Sel6}
\bysame, \emph{{Diophantine geometry over groups VI: The elementary theory of
  free groups}}, Geom. Funct. Anal. \textbf{16} (2006), 707--730.

\bibitem[Sta99]{Stallings}
John~R. Stallings, \emph{{Whitehead graphs on handlebodies}}, Geometric Group
  Theory Down Under (Canberra 1996), de Gruyter, Berlin, 1999, pp.~317--330.

\end{thebibliography}
\ \\ \\
University of Leeds,\\
School of Mathematics,\\
LS2 9JT, UK\\
rsklinos@maths.leeds.ac.uk

\end{document}